\documentclass[a4paper,11pt]{amsart}

\usepackage{amsmath,amssymb,amsthm,mathtools}
\usepackage{float}
\usepackage{color}
\usepackage[colorlinks=true,
            linkcolor=blue,
            citecolor=blue,
            urlcolor=blue,
            filecolor=blue,
            menucolor=blue,
            runcolor=blue]{hyperref}
\usepackage{enumitem}
\usepackage{diagbox}
\usepackage{makecell}
\allowdisplaybreaks

\numberwithin{equation}{section}
\theoremstyle{plain}
\newtheorem{thm}{Theorem}[section]

\newtheorem{lem}[thm]{Lemma}

\newtheorem{rem}{Remark}

\makeatletter

\allowdisplaybreaks[3]

\makeatother

\title[Minimal Willmore Hypersurfaces with Constant Scalar Curvature]
{Chern conjecture on Minimal Willmore Hypersurfaces with Constant Scalar Curvature}

\author[J. Q. Ge]{Jianquan Ge}
\address{$^{}$School of Mathematical Sciences, Laboratory of Mathematics and Complex Systems, Beijing Normal University, Beijing 100875, P. R. CHINA.}
\email{jqge@bnu.edu.cn}

\author[H. X. Tan]{Huixin Tan}
\address{$^{}$School of Mathematical Sciences, Laboratory of Mathematics and Complex Systems, Beijing Normal University, Beijing 100875, P. R. CHINA.}
\email{hxtan@mail.bnu.edu.cn}

\author[W. J. Yan]{Wenjiao Yan}
\address{$^{}$School of Mathematical Sciences, Laboratory of Mathematics and Complex Systems, Beijing Normal University, Beijing 100875, P. R. CHINA.}
\email{wjyan@bnu.edu.cn}

\author[Y. H. Zhang]{Yunheng Zhang$^{*}$}
\address{$^{*}$School of Mathematical Sciences, Laboratory of Mathematics and Complex Systems, Beijing Normal University, Beijing 100875, P. R. CHINA.}
\email{yunheng@mail.bnu.edu.cn}

\thanks{$^{*}$ the corresponding author.}
\thanks{The project is partially supported by NSFC (No. 12171037, 12271038, 12571049) and the Fundamental Research Funds for the Central Universities.}

\begin{document}

\begin{abstract}
	In this paper, we prove that for an $n$-dimensional closed minimal Willmore hypersurface $M^n$ with constant scalar curvature in the unit sphere $\mathbb{S}^{n+1}$, the squared norm $S$ of the second fundamental form of $M^n$ satisfies $S\geqslant n+\frac{4n+9-\sqrt{4 n^{2}+60 n+81}}{2}$ if $S>n$. This proves, in the approximate sense, the Chern conjecture about the second gap ($S\geqslant 2n$ if $S>n$), which will be fully verified under a further inequality condition about the 4-th mean curvature. 
\end{abstract}

\maketitle

\section{Introduction}

The renowned Chern conjecture was initially proposed by Chern \cite{Chern,CDK} in 1968 and 1970, and later included as the 105th problem in Yau's Problem Section \cite{Yau} in 1982. It can be
stated as follows.

\vspace{2mm}
\noindent
\textbf{Chern conjecture (Original version).}
\textit{
	Let $M^n$ be a closed, minimally immersed hypersurface of the unit
	sphere $\mathbb{S}^{n+1}$ with constant scalar curvature $R_M$. Then for each $n$, the set of all possible values of $R_M$ is discrete.
}
\vspace{2mm}

In 1968, J. Simons \cite{Sim} proved the following fundamental result, motivating the proposition of the Chern conjecture:

\vspace{2mm}
\noindent
\textbf{Simons theorem.}
\textit{
	Let $M^n$ be a closed, minimally immersed hypersurface of the unit sphere
	$\mathbb{S}^{n+1}$ . Then
	\begin{equation*}
		\int_{M}S(S-n)\geqslant 0,
	\end{equation*}
	where $S$ denotes the squared norm of the second fundamental form of $M^n$. In particular, if $0\leqslant S\leqslant n$, one has either $S\equiv0$ or $S\equiv n$ on $M^n$.
}
\vspace{2mm}

The classification of $S\equiv n$ in the Simons theorem was characterized by Chern-do Carmo-Kobayashi \cite{CDK} and Lawson \cite{Lawson} independently: The Clifford tori $\mathbb{S}^{k}\Big(\sqrt\frac{k}{n}\Big)\times \mathbb{S}^{n-k}\Big(\sqrt\frac{n-k}{n}\Big)$, $1\leqslant k\leqslant n-1$, are the only closed minimal hypersurfaces in $\mathbb{S}^{n+1}$ satisfying $S\equiv n$. Note that for minimal hypersurfaces in $\mathbb{S}^{n+1}$, the scalar curvature $R_{M}$ satisfies $S=n(n-1)-R_{M}$. Hence, the Simons theorem established the first gap of the original Chern conjecture. 
Furthermore, motivated by the fact that
the isoparametric hypersurfaces are the only known examples of minimal hypersurfaces with
constant scalar curvature in $\mathbb{S}^{n+1}$,
the original Chern conjecture has been reformulated into the following stronger version (cf. \cite{Verstraelen}).

\vspace{2mm}
\noindent
\textbf{Chern conjecture (Stronger version).}
\textit{
	Let $M^n$ be a closed, minimally immersed hypersurface of the unit
	sphere $\mathbb{S}^{n+1}$ with constant scalar curvature. Then $M^n$ is isoparametric.
}
\vspace{2mm}

Recall that a hypersurface in $\mathbb{S}^{n+1}$ is called isoparametric if all its principal curvature functions are constant.
In 1983, Peng-Terng pointed out that if $M^{n}$ is a compact minimal isoparametric hypersurface in $\mathbb{S}^{n+1}$ with $g$ distinct principal curvatures, then $S = (g-1)n$ (due to the well-known work of Münzner \cite{Munzner-1980,Munzner-1981}, $g$ can only be $1$, $2$, $3$, $4$, or $6$).
Moreover, Peng-Terng \cite{Peng-T1} initiated the study of the second pinching problem
for $S$ of minimal hypersurfaces in $\mathbb{S}^{n+1}$, and made the first breakthrough:

\vspace{2mm}
\noindent
\textbf{Peng-Terng theorem.}
\textit{
	Let $M^n$ be a closed, minimally immersed hypersurface in the unit
	sphere $\mathbb{S}^{n+1}$ with constant scalar curvature. If $S>n$, then $S>n+\frac{1}{12n}$.
}
\vspace{2mm}

In 1998, Yang-Cheng \cite{1} improved the second gap $\frac{1}{12n}$ to $\frac{n}{3}$, and their precise estimates and elegant techniques have had a profound influence and will play a key role in our proof. In 2007, by introducing new parameters, Suh-Yang \cite{2007} further improved it to $\frac{3n}{7}$. It is also known as Chern conjecture that the ideal second gap is $n$, namely, $S\geqslant 2n$ if $S>n$ (cf. \cite{Cheng-Wei}).

Denote by $f_k$ $(1\leqslant k\leqslant n)$ the sum of $k$-th power of the principal curvatures. Under the additional assumption that $f_3$ is constant, the second gap of $S$ can be improved to $\frac{2n}{3}$ by Yang-Cheng \cite{1}. Recently, Cheng-Wei-Yamashiro \cite{2} made a significant improvement:

\vspace{2mm}
\noindent
\textbf{Theorem (\cite{2}).}
\textit{
	Let $M^n\;(n\geqslant 5)$ be a complete minimal hypersurface in $\mathbb{S}^{n+1}$ with constant scalar curvature. If $f_3$ is constant and
	$S>n$, then
	\begin{equation*}
		S>1.8252n-0.712898.
	\end{equation*} 
}

Considering the stronger version of the Chern conjecture, a number of significant partial results have been established over the past decades. In 1993, building on the seminal works of Cartan, Simons, Chern-do Carmo-Kobayashi \cite{CDK}, and Peng-Terng \cite{Peng-T1}, Chang \cite{Chang} verified the stronger Chern conjecture in the case $n=3$. As a matter of fact, de Almeida and Brito \cite{A-B} proved that the stronger Chern conjecture holds for hypersurfaces $M^3 \subset \mathbb{S}^4$ with constant mean curvature and nonnegative constant scalar curvature. As emphasized in \cite{Scherfner-Weiss-Yau-2012}, the method developed in \cite{A-B} has since become a fundamental conceptual framework for a large body of subsequent research in this direction.
Recently, an important development was achieved by Tang-Wei-Yan \cite{T-W-Y} and Tang-Yan \cite{T-Y}, who completely generalized this result of \cite{A-B} to arbitrary dimensions. For further background, developments, and related results on the Chern conjecture, we refer the reader to the surveys \cite{Cheng-Wei, GQTY, GT, Scherfner-Weiss-Yau-2012}.


In \cite{DGW}, Deng-Gu-Wei proved that any closed minimal Willmore hypersurface $M^4\subset \mathbb{S}^{5}$
with constant scalar curvature must be isoparametric. In this paper, inspired by the methods of \cite{1}, we obtain the second gap of $S$ for closed minimal Willmore hypersurfaces in the unit sphere.
\begin{thm}\label{main}
    Let $M^n$ $(n\geqslant5)$ be an n-dimensional closed minimal Willmore hypersurface in $\mathbb{S}^{n+1}$ with constant scalar curvature. If $S>n$, then
    \[
        S\geqslant n+\frac{4n+9-\sqrt{4 n^{2}+60 n+81}}{2}=2n-T(n),
    \]
    where
    \[
        T(n) =
        \frac{\sqrt{4 n^{2}+60 n+81} - 2n - 9}{2}
        = 3-\frac{18}{n}+\frac{135}{n^{2}}-\frac{2349}{2n^{3}}+\frac{44955}{4n^{4}}+O\left(\frac{1}{n^{5}}\right).
    \]
\end{thm}

\begin{rem}
    Notice that $T(n)$ reflects the approximation accuracy relative to $2n$. Besides, we also see that
    \[
        n+\frac{4n+9-\sqrt{4 n^{2}+60 n+81}}{2}>1.8252n-0.712898.
    \]
    Specifically, we compare the lower bounds of $S$ given by Theorem \ref{main} and \cite[Theorem 1.2]{2} for several values of $n$:
    \begin{table}[H]\centering
		\begin{tabular}{cccccccc}
            \hline
			$n$ & $5$ & $10$ & $25$ & $50$ & $100$ & $1000$ & $10000$ \\
			\Xhline{1pt}
			\textup{Theorem \ref{main}} & $8.534$ & $18.060$ & $47.558$ & $97.313$ & $197.167$ & $1997.017$ & $19997.001$ \\
			\textup{Theorem (\cite{2})} & $8.413$ & $17.539$ & $44.917$ & $90.547$ & $181.807$ & $1824.487$ & $18251.287$ \\
            \hline
		\end{tabular}
	\end{table}
\end{rem}

In the proof of Theorem \ref{main}, the upper bound estimate of $f_4$ at one point (in Lemma \ref{imp}) will be obtained. Consequently, if $f_4$ further has a proper lower bound, one can achieve the goal ``$S\geqslant 2n$".

\begin{thm}\label{f4}
	Let $M^n$ be an n-dimensional closed minimal Willmore hypersurface in $\mathbb{S}^{n+1}$ with constant scalar curvature. Suppose that $S>n$ and one of the following conditions are satisfied: $(1)~ f_{4}\geqslant3S$, $(2)~ f_{4}\geqslant\dfrac{3S^2}{2n}$. Then $S\geqslant2n$.
\end{thm}

\begin{rem}
    Both inequality conditions (1) (2) in Theorem \ref{f4} are satisfied for all minimal isoparametric hypersurfaces in $\mathbb{S}^{n+1}$ with $g\geqslant3$. 
\end{rem}

\section{Preliminaries}
We consider an $n$-dimensional closed minimal hypersurface $M^n$ immersed in the unit sphere $\mathbb{S}^{n+1}$. 
To facilitate our discussion,
we choose a local orthonormal frame 
$\{e_1,\ldots,e_{n+1}\}$ on $\mathbb{S}^{n+1}$ such that $\{e_1,\ldots,e_{n}\}$ forms a tangent frame of $M^n$ and $e_{n+1}$ is the unit normal vector field of $M^n$. We assume the range of the indices as follows:
\begin{equation*}
	1\leqslant i,j,k,l,\ldots\leqslant n.
\end{equation*}
Let $\{\omega_1,\ldots,\omega_{n+1}\}$ denote the dual coframe, 
so that restricted to $M^n$, 
\[
    \omega_{n+1} = 0,
    \qquad
    d\omega_{n+1}
    = \sum_i \omega_{n+1,i}\wedge\omega_i = 0.
\]

By Cartan's lemma, one has
\[
    \omega_{n+1,i} = \sum_{j} h_{ij}\,\omega_j,\qquad h_{ij}=h_{ji},
\]
where $h_{ij}$ are the coeffients of the second fundamental form $h=\sum_{i,j}h_{ij}\omega_{i}\omega_{j}$ of $M^n$.
Moreover, the mean curvature $H$ and the squared norm of $h$ take the form
\[
    H=\frac{1}{n}\sum_{i}h_{ii},
    \qquad
    S = |h|^2 = \sum_{i,j} h_{ij}^2.
\]
Denote the covariant derivatives of the second fundamental form with respect to the Levi-Civita connection on $M^n$ by 
\[
    h_{ijk} = \nabla_k h_{ij}, 
    \qquad
    h_{ijkl} = \nabla_l\nabla_k h_{ij}.
\]
The structure equations then yield the Gauss equation, the Codazzi equation, and the Ricci formula:
\begin{align}
    R_{ijkl} 
      &= \delta_{ik}\delta_{jl} - \delta_{il}\delta_{jk}
         + h_{ik}h_{jl} - h_{il}h_{jk}, \label{equ-gauss} \\
    h_{ijk} &= h_{ikj}, \notag \\
    h_{ijkl} - h_{ijlk}
      &= \sum_m h_{im}R_{mjkl}
         + \sum_m h_{mj}R_{mikl}. \label{equ-ricci}
\end{align}

In what follows, we adopt a local orthonormal frame in which the second fundamental form is diagonal, namely $h_{ij}=\lambda_i\,\delta_{ij}$.
Using the assumption that $M^{n}$ is minimal, together with the Ricci formula, we derive the classical identity
\[
    \Delta h_{ij}
    = \sum_{k} h_{ijkk}
    = (n-S)h_{ij}.
\]

If $S$ is constant, then a direct computation yields
\[
    \sum_{i,j,k} h_{ijk}^{2}
    = S(S-n).
\]
Furthermore,
a more detailed computation shows that
\begin{equation}\label{ss}
    \sum_{i,j,k,l} h_{ijkl}^{2}
    = S(S-n)(S-2n-3) + 3(A-2B),
\end{equation}
where $A$ and $B$ are given explicitly by
$$A := \sum_{i,j,k,l,m} h_{ijk}\,h_{ijl}\,h_{km}\,h_{ml}
= \sum_{i,j,k}\lambda_i^{\,2}\,h_{ijk}^{2},$$
$$B := \sum_{i,j,k,l,m} h_{ijk}\,h_{ilm}\,h_{jl}\,h_{km}
= \sum_{i,j,k} \lambda_i \lambda_j\, h_{ijk}^{2}.$$
Define
\begin{align*}
    f_3 &= \sum_{i,j,k} h_{ij}h_{jk}h_{ki}
         = \sum_{i} \lambda_i^{3},\\
    f_4 &= \sum_{i,j,k,l} h_{ij}h_{jk}h_{kl}h_{li}
         = \sum_{i} \lambda_i^{4},
\end{align*}
which correspond to higher-order traces of the shape operator and are encountered naturally in subsequent estimates.

Next, for each collection of indices, define the fully symmetrized four-index symbol
\begin{equation}\label{uijkl}
    u_{ijkl}
    := \frac{1}{4}\bigl(h_{ijkl}+h_{jkli}+h_{klij}+h_{lijk}\bigr).
\end{equation}


\section{Proof of Theorem \ref{main}}
To prove Theorem \ref{main}, we first establish several lemmas.

\begin{lem}\label{new}
    Let $M^{n}$ be a closed minimal hypersurface in $\mathbb{S}^{n+1}$, then
    \begin{equation}\label{ssu}
        \sum_{i,j,k,l} h_{ijkl}^{2}
        = \sum_{i,j,k,l} u_{ijkl}^{2}
          + \frac{3}{2}\bigl(Sf_{4}-f_{3}^{2}-2S^{2}+nS\bigr).
    \end{equation}
\end{lem}

\begin{proof}
Squaring equation \eqref{uijkl} and summing over all indices gives
\begin{align*}
    \sum_{i,j,k,l} u_{ijkl}^{2}
    &= \frac{1}{16}\sum_{i,j,k,l}
        \bigl(h_{ijkl} + h_{jkli} + h_{klij} + h_{lijk}\bigr)^{2} \\
    &= \frac{1}{4}\sum_{i,j,k,l} h_{ijkl}^{2}
       + \frac{3}{4}\sum_{i,j,k,l} h_{ijkl}\,h_{ijlk}.
\end{align*}
From the Ricci formula \eqref{equ-ricci}, we obtain
\[
    \frac{3}{4}\sum_{i\neq j}(h_{ijij}-h_{jiji})^{2}
    = \frac{3}{2}\bigl(Sf_{4}-f_{3}^{2}-2S^{2}+nS\bigr).
\]
Consequently,
\begin{align*}
    &\sum_{i,j,k,l} h_{ijkl}^{2}
      - \sum_{i,j,k,l} u_{ijkl}^{2}
      - \frac{3}{4}\sum_{i\neq j}(h_{ijij}-h_{jiji})^{2} \\
    =&\frac{3}{4}\sum_{i,j,k,l}h_{ijkl}^{2}-\frac{3}{4}\sum_{i,j,k,l}h_{ijkl}h_{ijlk}-\frac{3}{4}\sum_{i\neq j}(h_{ijij}-h_{jiji})^2 \\
    =&\frac{3}{8}\sum_{i,j,k,l}(h_{ijkl}^{2}+h_{ijlk}^{2})-\frac{3}{4}\sum_{i,j,k,l}h_{ijkl}h_{ijlk}-\frac{3}{4}\sum_{i\neq j}(h_{ijij}-h_{jiji})^2 \\
    =& \frac{3}{8}\sum_{i,j,k,l}(h_{ijkl}-h_{ijlk})^{2}
       - \frac{3}{4}\sum_{i\neq j}(h_{ijij}-h_{jiji})^{2}.
\end{align*}

A rearrangement of index pairs shows that
\[
    \sum_{i,j,k,l}(h_{ijkl}-h_{ijlk})^{2}
    = \sum_{\substack{(i,j)\neq(k,l)\\(i,j)\neq(l,k)}}(h_{ijkl}-h_{ijlk})^{2}
      + 2\sum_{i\neq j}(h_{ijij}-h_{jiji})^{2}.
\]
Hence
\begin{equation}\label{eqn}
    \sum_{i,j,k,l} h_{ijkl}^{2}
- \sum_{i,j,k,l} u_{ijkl}^{2}-\frac{3}{4}\sum_{i\neq j}(h_{ijij}-h_{jiji})^{2}
= \frac{3}{8}\sum_{\substack{(i,j)\neq(k,l)\\(i,j)\neq(l,k)}}(h_{ijkl}-h_{ijlk})^{2}.
\end{equation}

For all index pairs satisfying $(i,j)\neq(k,l)$ and $(i,j)\neq(l,k)$, the Gauss equation \eqref{equ-gauss} and the Ricci identity \eqref{equ-ricci} imply
\[
    h_{ijkl}-h_{ijlk}
    = (\lambda_i - \lambda_j)(1+\lambda_i\lambda_j)
      (\delta_{ik}\delta_{jl} - \delta_{il}\delta_{jk})
    = 0.
\]
Thus the right‑hand side of \eqref{eqn} vanishes, and the desired identity follows. 

\end{proof}

To simplify the computations below, we define
\[
    f := f_{4} - \frac{1}{S}f_{3}^{2} - \frac{S^{2}}{n}.
\]
In case that $M^n$ is a minimal Willmore hypersurface, $f_3$ is forced to be $0$, consequently
\begin{equation}\label{f}
    f = f_{4} - \frac{S^{2}}{n}.
\end{equation}

Set
\[
    t=\frac{S-n}{S},
\]
so that $0 < t < 1$.

\begin{lem}\label{ab}
Let $M^{n}$ be a closed minimal Willmore hypersurface of $\mathbb{S}^{n+1}$ with constant scalar curvature.  
Then there exists a point $x_{0}\in M^{n}$ at which 
\begin{equation}\label{AandB}
    \begin{aligned}
        A &= \left(\frac{2}{5}t + \frac{1}{5}\right) S f
           + \frac{3tS^{2}}{5(1-t)},\\
        B &= \left(\frac{1}{5}t - \frac{2}{5}\right) S f
           - \frac{tS^{2}}{5(1-t)}.
    \end{aligned}
\end{equation}
\end{lem}

\begin{proof}
The minimal Willmore condition forces $f_{3}=0$. From
\[
    \Delta f_{3}
    = -3(S-n)f_{3} + 6\sum_{i,j,k}\lambda_i h_{ijk}^{2},
\]
we immediately obtain
\begin{equation}\label{C0-used}
    \sum_{i,j,k} \lambda_i h_{ijk}^{2} = 0.
\end{equation}
A similar computation for $f_4$ gives
\[
    \Delta f_{4}
    = -4(S-n)f_{4} + 4(2A+B).
\]
Integrating $\Delta f_{4}$ over $M^{n}$ yields
\[
    \int_{M} -4(S-n)f_{4} + 4(2A+B) = 0.
\]
By continuity, there exists a point $x_{0}\in M^{n}$ such that at $x_0$, 
\[
    2A + B = tS f_{4}.
\]
Substituting \eqref{f} into the this equality, we obtain
\begin{equation}\label{2A+B}
    2A+B
    = tS f + \frac{tS^{2}}{1-t}.
\end{equation}

On the other hand, direct expansion shows
\begin{equation*}
	\begin{aligned}
		0&=\frac{1}{2}\sum_{i,j,k}S_{ij}h_{jk}h_{ki}\\
		&=\sum_{k}\lambda_{k}^{2}\Big(\sum_{i}h_{iikk}\lambda_{i}+\sum_{i,j}h_{ijk}^2\Big)\\
		&=\sum_{i,k}h_{iikk}\lambda_{k}^2\lambda_{i}+\sum_{i,j,k}\lambda_{k}^2h_{ijk}^2\\
		&=\sum_{i,k}\Big(h_{kkii}+(\lambda_{i}-\lambda_{k})(1+\lambda_{i}\lambda_{k})\Big)\lambda_{k}^2\lambda_{i}+A\\
		&=\sum_{i}\Big(\frac{1}{3}(f_{3})_{ii}-2\sum_{j,k}\lambda_{k}h_{ijk}^2\Big)\lambda_{i}+\sum_{i,k}(\lambda_{i}-\lambda_{k})(1+\lambda_{i}\lambda_{k})\lambda_{k}^2\lambda_{i}+A\\
		&=S^2-Sf_{4}+(A-2B).
	\end{aligned}
\end{equation*}
Hence, 
\begin{equation}\label{A-2B}
    A - 2B
    = S f_{4} - S^{2}
    = S f + \frac{tS^{2}}{1-t}.
\end{equation}

Solving the linear system formed by \eqref{2A+B} and \eqref{A-2B} gives precisely the expressions in \eqref{AandB}.

\end{proof}

\begin{lem}\label{imp}
    Let $M^{n}$ be a closed minimal Willmore hypersurface of $\mathbb{S}^{n+1}$ with constant scalar curvature.  
    Then at the point $x_{0}$ in Lemma~\ref{ab}, the inequality
    \begin{equation}\label{equ-f-upper}
        f \leqslant 
        \frac{t S}{(1-t)(3-4t)}
    \end{equation}
    holds for $0<t<\dfrac{3}{4}$.
\end{lem}

\begin{proof}
From Lemma~\ref{ab}, we obtain
\[
    A + 2B
    = \left(\frac{4}{5}t - \frac{3}{5}\right) S f
      + \frac{tS^{2}}{5(1-t)}.
\]
A direct expansion gives
\begin{equation*}
	\begin{aligned}
		0&\leqslant\frac{1}{3}\sum_{i,j,k}\big((\lambda_{i}+\lambda_{j}+\lambda_{k})h_{ijk}\big)^2\\
		&=\frac{1}{3}\sum_{i,j,k}(\lambda_{i}^2+\lambda_{j}^2+\lambda_{k}^2+2\lambda_{i}\lambda_{j}+2\lambda_{i}\lambda_{k}+2\lambda_{j}\lambda_{k})h_{ijk}^2\\
		&=A+2B.
	\end{aligned}
\end{equation*}
Hence,
\[
    \left(\frac{4}{5}t - \frac{3}{5}\right) S f
      + \frac{tS^{2}}{5(1-t)} \geqslant 0.
\]
When $0<t<\dfrac{3}{4}$, solving the inequality for $f$ yields precisely \eqref{equ-f-upper}.

\end{proof}




Now, we are ready to give a proof of Theorem \ref{main}.
\vspace{2mm}

\noindent
\textbf{\textit{Proof of Theorem \ref{main}.}}  ~~
We firstly establish the lower bound $$t \geqslant \frac{10 n+15-\sqrt{4 n^{2}+60 n+81}}{16n+24}.$$

Using the notation $t=\dfrac{S-n}{S}$, Yang and Cheng (\cite{1}) proved that if $f_{3}$ is constant, then 
\[
    \sum_{i,j,k,l}u_{ijkl}^2
    \geqslant 
    \frac{1}{2Sf}\left(Sf-2A+tf_{3}^2+\frac{2tS^2}{1-t}\right)^2
    +\frac{1}{S^2}\left(\sum_{i,j,k}\lambda_{i}h_{ijk}^2\right)^2
    +\frac{t^2 S^2}{2(1-t)}.
\]
Since $f_{3}=0$ for a minimal Willmore hypersurface, combining \eqref{C0-used}, this simplifies to
\begin{equation}\label{key-f3zero}
    \sum_{i,j,k,l}u_{ijkl}^2
    \geqslant 
    \frac{1}{2Sf}\left(Sf-2A+\frac{2tS^2}{1-t}\right)^2
    +\frac{t^2 S^2}{2(1-t)}.
\end{equation}

On the other hand, combining \eqref{ss} with Lemma~\ref{new}, we derive that
\begin{equation*}
\begin{aligned}
    \sum_{i,j,k,l}u_{ijkl}^{2}
    &= S(S-n)(S-2n-3) + 3(A-2B)
       -\frac{3}{2}\left(Sf_{4}-2S^{2}+nS\right)\\
    &= tS^{2}(S-2n-3) + 3(A-2B)
       -\frac{3}{2}\left(Sf+\frac{S}{n}(S-n)^{2}\right).
\end{aligned}
\end{equation*}
Consequently,
\begin{equation}\label{equ-main}
\begin{aligned}
    &tS^2(2tS-S-3)
    +3(A-2B)
    -\frac{3}{2}\left(Sf+\frac{t^2S^2}{1-t}\right) \\
    \geqslant&~
    \frac{1}{2Sf}\left(Sf-2A+\frac{2tS^2}{1-t}\right)^2
    +\frac{t^2 S^2}{2(1-t)}.
\end{aligned}
\end{equation}

Substituting the expressions \eqref{AandB} for $A$ and $B$ from Lemma~\ref{ab} into \eqref{equ-main}, we obtain
\begin{equation}\label{ineq-f}
    \left(8 t^{2}-12 t-33\right) f
    +\frac{8 S^{2} t^{2}}{(t-1)^{2} f}
    \leqslant
    \frac{\left(50 t^{2} S-75 t S+25 S-41 t+12\right)tS}{t-1}.
\end{equation}

Suppose that 
\[
    0 < t < \frac{10 n+15-\sqrt{4 n^{2}+60 n+81}}{16n+24}.
\]
Clearly, $0<t<\dfrac{3}{4}$, hence
\[
    8 t^{2}-12 t-33
    =\left(t-\frac{3}{4}-\frac{5\sqrt{3}}{4}\right)\!
     \left(t-\frac{3}{4}+\frac{5\sqrt{3}}{4}\right)
    <0.
\]
Notice that the coefficient of $f^{-1}$ in \eqref{ineq-f} is positive.  Using the upper bound in (\ref{equ-f-upper}), inequality \eqref{ineq-f} becomes
\[
\begin{aligned}
    &\left(8 t^{2}-12 t-33\right)
    \frac{t S}{(1-t)(3-4t)}
    +\frac{8 S^{2} t^{2}}{(t-1)^{2}}
     \frac{(1-t)(3-4t)}{t S}\\
     \leqslant
     &~\frac{\left(50 t^{2} S-75 t S+25 S-41 t+12\right)tS}{t-1}.
\end{aligned}
\]
Simplifying,
\[
\begin{aligned}
    \frac{t S (136 t^{2}-204 t+39)}{(1-t)(3-4 t)}
    \leqslant
    \frac{(50 t^{2} S-75 t S+25 S-41 t+12)tS}{t-1}.
\end{aligned}
\]
Substituting $S=\dfrac{n}{1-t}$ into the inequality above, we get a simpler form
\[
\begin{aligned}
    \frac{136 t^{2}-204 t+39}{(1-t)(3-4 t)}
    \leqslant
    \frac{(50 n+41)t-25 n-12}{1-t}.
\end{aligned}
\]
Because $(1-t)(3-4t)>0$, this is equivalent to
\[
    136 t^{2}-204 t+39
    \leqslant
    -\bigl(50 n t-25 n+41 t-12\bigr)(4 t-3),
\]
that is,
\[
    (8 n+12)t^{2}-(10 n+15)t+3n+3\leqslant 0.
\]
Solving the quadratic inequality gives
\[
    \frac{10 n+15-\sqrt{4 n^{2}+60 n+81}}{16n+24}
    \leqslant t
    \leqslant
    \frac{10 n+15+\sqrt{4 n^{2}+60 n+81}}{16n+24},
\]
which contradicts our assumption on $t$. Therefore, $$t\geqslant\dfrac{10 n+15-\sqrt{4 n^{2}+60 n+81}}{16n+24}.$$

Finally, using $t=\dfrac{S-n}{S}$, we obtain
\[
    S=\frac{n}{1-t}
    \geqslant 
    n+\frac{4n+9-\sqrt{4 n^{2}+60 n+81}}{2}
    =2n-T(n),
\]
where
\[
\begin{aligned}
    T(n)
    &=\frac{\sqrt{4 n^{2}+60 n+81}-2n-9}{2}
     = 3-\frac{18}{n}+\frac{135}{n^{2}}
      -\frac{2349}{2n^{3}}
      +\frac{44955}{4n^{4}}
      +O\!\left(\frac{1}{n^{5}}\right).
\end{aligned}
\]

This completes the proof of Theorem \ref{main}.
\hfill$\Box$

\section{Proof of Theorem \ref{f4}}

\begin{proof}
	Suppose that $0<t<\dfrac{1}{2}$. By Lemma \ref{imp}, we have
	\begin{equation*}
		f_{4}\leqslant\frac{tS}{(1-t)(3-4t)}+\frac{S}{1-t}=\frac{3S}{3-4t}.
	\end{equation*} 
	Combining this with either condition (1) or (2) in Theorem \ref{f4}, we deduce that $t\geqslant \dfrac{1}{2}$, which leads to a contradiction. 
	
	Hence, $t\geqslant\dfrac{1}{2}$ and consequently $S=\dfrac{n}{1-t}\geqslant 2n$.
	
\end{proof}

\vspace{3mm}

\section{Appendix}

In addition, to illustrate how the lower bound behaves from a more intuitive perspective, we conclude that:
\begin{enumerate}[label=$\bullet$, itemsep=-2pt, topsep=0pt]
    \item If $n\geqslant 5$,
        $S\geqslant n + 0.7068287800n$;
    \item If $n\geqslant 25$,
        $S\geqslant n + 0.9023458997n$;
    \item If $n\geqslant 100$,
        $S\geqslant n + 0.9716757246n$;
    \item If $n\geqslant 10000$,
        $S\geqslant n + 0.9997001798n$.
\end{enumerate}

\vspace{3mm}



\begin{thebibliography}{99} 
\bibitem{A-B}S. C. de Almeida and F. G. B. Brito, \textit{Closed 3-dimensional hypersurfaces with constant mean curvature and constant scalar curvature}, Duke Math. J., \textbf{61} (1990), 195-206.

\bibitem{Chang}S. P. Chang, \textit{On minimal hypersurfaces with constant scalar curvature in $S^4$}, J. Differential Geom., \textbf{37} (1993), 523-534.

\bibitem{Chern} 
S. S. Chern, \textit{Minimal submanifolds in a Riemannian manifold}, (mimeographed), University of Kansas, Lawrence, 1968.

\bibitem{CDK}
S. S. Chern, M. do Carmo and S. Kobayashi, \textit{Minimal
	submanifolds of a sphere with second fundamental form of constant
	length}, In: Functional Analysis and Related Fields, Springer, New York,
1970, 59-75.

\bibitem{Cheng-Wei}
Q. M. Cheng and G. X. Wei, \textit{Chern problems and Chern conjecture for minimal hypersurfaces} (in Chinese), Sci. Sin. Math., \textbf{55} (2025), 131-144.

\bibitem{2}
Q. M. Cheng, G. X. Wei and T. Yamashiro, \textit{The second gap of the scalar curvature of complete minimal hypersurfaces}, Commun. Anal. Geom., \textbf{33} (2025), 623-636. 

\bibitem{DGW}
Q. T. Deng, H. L. Gu and Q. Y. Wei, \textit{Closed Willmore minimal hypersurfaces with
constant scalar curvature in $\mathbb{S}^{5}(1)$ are isoparametric}, Adv. Math., \textbf{314} (2017), 278-305.

\bibitem{GQTY}
J. Q. Ge, C. Qian, Z. Z. Tang and W. J. Yan, \textit{An overview of the development of isoparametric theory} (in Chinese), Sci. Sin. Math., \textbf{55} (2025): 145-168.

\bibitem{GT}
J. Q. Ge and Z. Z. Tang, \textit{Chern conjecture and isoparametric hypersurfaces}, Differential Geometry, Advanced Lectures in Mathematics, vol. 22, International Press, Somerville, MA, 2012, 49-60.

\bibitem{Lawson}
H. B. Lawson, \textit{Local rigidity theorems for minimal hypersurfaces}, Ann. of Math., \textbf{89} (1969), 187-197.

\bibitem{Munzner-1980}
H. F. Münzner, \textit{Isoparametrische Hyperfl{\"a}chen in Sph{\"a}ren, I}, Math. Ann., \textbf{251} (1980), 57-71.

\bibitem{Munzner-1981}
H. F. Münzner, \textit{Isoparametrische Hyperfl{\"a}chen in Sph{\"a}ren, II}, Math. Ann., \textbf{256} (1981), 215-232.

\bibitem{Peng-T1}
C. K. Peng and C. L. Terng, \textit{Minimal hypersurfaces of spheres
with constant scalar curvature}, In: Seminar on minimal submanifolds,
Ann. of Math. Stud., vol. 103, Princeton Univ. Press, Princeton,
NJ, 1983, 177-198.

\bibitem{Scherfner-Weiss-Yau-2012}
M. Scherfner, S. Weiss and S. T. Yau, \textit{A review of the Chern conjecture for isoparametric hypersurfaces in spheres}, Advances of Geometric Analysis, Advanced Lectures in Mathematics, vol. 21, International Press, Somerville, MA, 2012, 175-187.

\bibitem{Sim}
J. Simons, \textit{Minimal varieties in Riemannian manifolds}, Ann.
of Math., \textbf{88} (1968), 62-105.

\bibitem{2007}
Y. J. Suh and H. Y. Yang, \textit{The scalar curvature of minimal hypersurfaces in a unit sphere}, Commun. Contemp. Math., \textbf{9} (2007),	183-200.

\bibitem{T-W-Y}Z. Z. Tang, D. Y. Wei and W. J. Yan, \textit{A sufficient condition for a hypersurface to be isoparametric}, Tohoku Math. J. (2), \textbf{72} (2020), 493-505.

\bibitem{T-Y}Z. Z. Tang and W. J. Yan, \textit{On the Chern conjecture for isoparametric hypersurfaces}, Sci. China Math., \textbf{66} (2023), 143-162.

\bibitem{Verstraelen}
L. Verstraelen, \textit{Sectional curvature of minimal submanifolds}, In: Proceedings Workshop on Differential Geometry (ed. S. Robertson et al.), Univ. Southampton, 1986, 48-62.

\bibitem{1}
H. C. Yang and Q. M. Cheng, \textit{Chern's conjecture on minimal
	hypersurfaces}, Math. Z., \textbf{227} (1998), 377-390.

\bibitem{Yau}    
S. T. Yau, \textit{Problem section}, In: Seminar on Differential Geometry, Ann. of Math. Stud., vol. 102, Princeton Univ. Press, Princeton, NJ, 1982, 669-706.
 
\end{thebibliography}
\end{document}